\documentclass[a4paper,10pt]{article}
\usepackage{amsmath}
\usepackage{amssymb}
\usepackage{amsthm}
\usepackage[latin1]{inputenc}

\newcommand{\R}{\mathbb{R}}
\newcommand{\C}{\mathbb{C}}
\newcommand{\N}{\mathbb{N}}
\newcommand{\Hy}{\mathbb{H}}
\newcommand{\vo}{\mathrm{dvol}}
\newcommand{\ov}{\overline}
\newcommand{\lm}{\lambda^+_{min}}
\renewcommand{\Re}{\mathrm{Re}}
\newcommand{\s}{\mathrm{scal\,}}
\renewcommand{\d}{\mathrm{d}}

\newcommand{\bq}{\begin{equation}}
\newcommand{\bqw}{\begin{equation*}}
\newcommand{\eq}{\end{equation}}
\newcommand{\eqw}{\end{equation*}}
\newcommand{\vol}{\mathrm{vol}}

\newtheorem{lem}{Lemma}
\newtheorem{cor}[lem]{Corollary}
\newtheorem{defi}[lem]{Definition}
\newtheorem{thm}[lem]{Theorem}

\theoremstyle{definition}
\newtheorem{ex}[lem]{Example}
\newtheorem{rem}[lem]{Remark}
\date{}
\title{Solutions of the equation of a spinorial Yamabe-type problem on manifolds of bounded geometry
\footnotetext{Acknowledgement: The main part of this work during the author's stay at the university of Regensburg. The author wants to thank Bernd Ammann for many fruitful discussions.}}
\author{Nadine Gro\ss e}
\begin{document}

\maketitle
\begin{abstract}
We consider a spinorial Yamabe-type problem on open manifolds of bounded geometry. The aim is to study the existence of solutions to the associated Euler-Lagrange-equation. We show that under suitable assumptions such a solution exists. As an application, we prove that existence of a solution implies the conformal Hijazi inequality for the underlying spin manifold.
\end{abstract}

\section{Introduction}
Let $(M,g,\sigma)$ be an $n$-dimensional connected Riemannian spin manifold with corresponding classical Dirac operator $D=D_g$. The spin structure $\sigma$ is supposed to be fixed.\\
The Dirac operator has a similar behaviour under conformal transformations as the conformal Laplacian, that is used to analyze the Yamabe invariant, see e.g. \cite{LP}. Thus, in the spirit of the Yamabe invariant, a conformal invariant of the Dirac operator, called $\lm$ was examined in \cite{Amm03a, Ammha}.
In \cite{NG} we generalized these considerations to open manifolds.\\

Many of the properties of the Yamabe invariants could be proven for $\lm$ as well, e.g. the value for the standard sphere is the highest possible one, cf. Theorem \ref{subset}. But other questions are still open, e.g. whether there exists closed spin manifolds, not conformally diffeomorphic to the standard sphere (with the same dimension) but with the $\lm$-invariant of this standard sphere.\\
The methods that can be used are sometimes similar to the ones of the Yamabe invariant, but since we work with spinors there is e.g. no maximum principle.\\

Firstly, a definition of $\lm$ is given by:
 
\begin{defi}
\[\lm(M,g,\sigma)=\inf \left\{ \frac{\Vert D\phi \Vert_{q_{crit}}^2}{(D\phi,\phi)}\ \Big|\ \phi\in C_c^\infty(M,S), (D\phi,\phi)>0 \right\}\]
with $q_{crit}=\frac{2n}{n+1}$ and $C_c^\infty(M,S)$ denotes the set of compactly supported smooth spinors.
\end{defi}

For $\ov{g}=f^2g$ ($f\in C^\infty_{>0}(M)$) and $\ov{\phi}=f^{-\frac{n-1}{2}}\phi$ (using the identification of spinor bundles w.r.t. conformal metrics \cite[Sect. 4.1]{Hija}) we have $\int_M \langle D_g\phi, \phi\rangle \vo_g=\int_M \langle D_{\ov{g}}\ov{\phi}, \ov{\phi}\rangle \vo_{\ov{g}}$\ and $\int_M |D_g\phi|^{q_{crit}}\vo_g = \int_M |D_{\ov{g}}\ov{\phi}|^{q_{crit}}\vo_{\ov{g}}$. Thus, $\lm$ is actually a conformal invariant.\\
If it is clear from the context to which metric we refer, we will shortly write $\lm(M)$.\\

The corresponding Euler-Lagrange-equation is 
\bq D\phi=\lm |\phi|^{p_{crit}-2}\phi \textrm{\ with\ } \Vert \phi \Vert_{p_{crit}}=1,\label{EL}\eq
where $p_{crit}=\frac{2n}{n-1}$, \cite[Lem. 2.7]{Amm03a}. On closed manifolds, the existence of a solution of \eqref{EL} was shown in \cite[Thm. 1.6.]{Amm03a} for $\lm (M)<\lm (S^n)$ using the compactness of the subcritical Sobolev embeddings.\\
On open manifolds, such Sobolev embeddings do not exist in general or if they exist, they aren't always compact.\\

The aim of this paper is to prove the existence of a solution under suitable assumptions. The idea is to adapt techniques we used in \cite{NG3} to prove the existence of a solution of the Yamabe equation to the spinorial case. We will use weighted Sobolev spaces, where compactness holds for manifolds of bounded geometry, i.e. manifolds where the injectivity radius is positive and the curvature and its derivatives of all orders are bounded. The weighted Sobolev embeddings for the spaces of spinors can be found in Appendix \ref{app}.

\begin{thm}\label{main1}
Let $(M^n,g,\sigma)$ be an open connected Riemannian spin manifold of bounded geometry satisfying $\ov{\lm(M)}>\lm (M)$. Moreover, we assume that there is an $L^{[q_{crit}, q_{crit}+\epsilon]}$-lower bound for an $\epsilon>0$. Then there exists a spinor $\phi\in H_1^{q_{crit}}\cap L^{p_{crit}}$ smooth outside its zero set and with $D\phi=\lm |\phi|^{p_{crit}-2}\phi$ and $\Vert \phi\Vert_{p_{crit}}=1$. Moreover, $\phi$ is locally in $C^{1,\alpha}$.
\end{thm}

Here, $\ov{\lm}$ denotes an invariant of the manifold at infinity, see Def. \ref{lminf}.

\begin{defi} We say that there is an {\emph{$L^q$-lower bound}}, if there exists a constant $c_q>0$ such that \[\Vert \phi\Vert_q\leq c\Vert D\phi\Vert_q\] 
for all $\phi\in H_1^q\cap \mathrm{im}_{L^q}(D)$ where $\mathrm{im}_{L^q}(D)$ denotes the image of $D: L^q\to L^q$.\\
Moreover, if there is an $L^{\tilde{q}}$-lower bound for all $\tilde{q}\in [q, q']$, we say that there is an $L^{[q,q']}$-lower bound.\end{defi}

\begin{rem} Note that although we call it a lower bound we allow $L^q$-harmonic spinors.
\end{rem}

We will prove Theorem \ref{main1} in Section \ref{pr} by considering corresponding weighted subcritical problems. There, the $L^q$-lower bounds will ensure the positivity of $\lambda_q$, see Lemma \ref{pos}.

{\bf Question.} It would be nice to see whether having an $L^{q}$-lower bound implies an $L^{[q, q+\epsilon]}$-lower bound. Or whether invertibility of the Dirac operator acting on $L^q$ and $im_{L^q}D$, respectively, is an open property in $q$?\\
One knows e.g. that the $L^p$-spectrum of the standard Laplacian is independent of $p$ if the Ricci curvature is bounded from below and the volume has subexponential growth (cf. \cite{Stu93} or \cite{Char05}). Maybe one can hope for a similar result for the Dirac operator.\\

The assumption on the $\lm$-invariant at infinity can be dropped when considering homogeneous manifolds of positive scalar curvature, see Theorem \ref{main2}.\\

As an application we show in Section \ref{cohij} that existence of a solution of the Euler-Lagrange equation implies the conformal Hijazi inequality, an inequality that compares $\lm$ with the Yamabe invariant $Q$, cf. Section \ref{cohij}:

\begin{thm}\label{Hijsum}
Let $(M^n,g,\sigma)$ be an open connected Riemannian spin manifold of bounded geometry with an $L^{[q_{crit}, q_{crit}+\epsilon]}$-lower bound for an $\epsilon>0$. Moreover, either let $\ov{\lm(M)}>\lm (M)$ or let $M$ be homogeneous with positive scalar curvature (cf. Theorem \ref{main2}). Then, 
\[\lm(M,g,\sigma)^2\geq \frac{n}{4(n-1)} Q(M,g).\]
\end{thm}

The outline of the paper is as follows: First we collect in Section \ref{pre} some properties of~$\lm$. In Section \ref{pr} we will prove the Theorem \ref{main1} introducing corresponding weighted subcritical problems. In the last section we show that the existence of a solution of the Euler-Lagrange equation of $\lm$ ensures that the conformal Hijazi inequality holds for the underlying manifold.\\
The weighted Sobolev embeddings needed in Section \ref{pr} are explained in the Appendix.
 
\section{Preliminaries on $\lm$}\label{pre}

We shortly collect known properties of~$\lm$: 

\begin{thm}\label{subset}\ \!\cite{NG, NGDiss}
Let $\Omega_1\subset \Omega_2$ be open subsets of $(M,g,\sigma)$ equipped with the induced metric and spin structure. Then,
\[\lm(M)\leq \lm(\Omega_2)\leq \lm(\Omega_1)\leq \lm(S^n)=\frac{n}{2}\omega_n^{\frac{1}{n}},\]
where $S^n$ is the standard sphere with volume $\omega_n$.
\end{thm}

We will further need the notion of the Yamabe invariant at infinity. With the help of this invariant one can ensure convergence of a minimizing sequence in the critical problem.

\begin{defi}\ \!\cite[Def. 1.3]{NG}\label{lminf} Let $M$ be open, connected and complete. Fix $z\in M$. Denote by $B_R\subset M$ the ball around $z$ with radius $R$ w.r.t. the metric $g$. Then,
\[\ov{\lm(M,g,\sigma)}:=\lim_{R\to \infty} \lm(M\setminus B_R, g, \sigma).\] 
\end{defi}

Existence of the limit follows from Theorem \ref{subset}. Moreover, the definition is independent of the choice of $z$ and we have $\ov{\lm(M)}\geq \lm(M)$.

%%%%%%%%%%%%%%%%%%%%%%%%%%%%%%%%%%%%%%%%%%%%%
%%%%%%%%%%%%%%%%%%%%%%%%%%%%%%%%%%%%%%%%%%%%%
\section{The weighted subcritical problem}\label{pr}

In this section, we will prove Theorem \ref{main1}.\\

{\bf Strategy.} The method itself is similar to the one we used for the Yamabe equation \cite{NG3} but the proofs of the steps themselves have to be adapted to the $\lm$-invariant involving the Dirac operator.\\ The main idea is again to consider a weighted subcritical problem, see Definition \ref{defwsp}.\\
Firstly, in Lemma \ref{Dwei} we prove the existence of a solution to the weighted subcritical problem, see Def. \ref{defwsp} for $\alpha>0$ and $q>q_{crit}$.  Then, convergence is obtained in a two-step process: In Lemma \ref{Dconv}, a solution of the weighted critical problem ($\alpha>0$, $q=q_{crit}$) is obtained. After that, we obtain in Lemma \ref{Dinfmass} a solution to the (unweighted) critical problem ($\alpha=0$, $q=q_{crit}$).\\

We fix a point $z\in M$. % and define $\rho(x)=\exp (-r)$ \m{smooth?} where $r:=dist (x,z)$. The function $\rho$ will be used as weight function ($\rho$ is radial and admissible, cf. Def. \ref{weight}).
Let $\rho\in C^\infty (M)$ be a radial (w.r.t. $z\in M$) and admissible weight, cf. Def. \ref{weight}, such that $|\rho|\leq 1$, $\rho(x)\to 0$ as $r=dist (x,z)\to \infty$. We choose $\rho\sim e^{-r}$, see Rem. \ref{exwe}.ii.
 
\begin{defi}\label{defwsp} 
 Let
\[ \lambda^\alpha_q=\lambda^\alpha_q (M,g,\sigma):=\inf\left\{ \frac{\Vert \rho^{-\alpha}D\phi \Vert_q^2}{(D\phi,\phi)}\Big|\ \phi\in C_c^\infty(M,S),\ (D\phi,\phi)>0 \right\}\]
where $\alpha\geq 0$ and $q\in [q_{crit}=\frac{2n}{n+1},2)$. Moreover, $\lambda_q:=\lambda_q^{\alpha=0}$. 
\end{defi}

Note that  $\lm:=\lambda_{q_{crit}}^{\alpha=0}$.

\begin{rem}\hfill\label{remim}\\
\textbf{i)}\ \ If we decompose $C_c^\infty (M,S) \ni \phi=\phi_i+\phi_k$ such that $\phi_k\in \ker_{L^q} D$ and $\phi_i\in im_{L^q} D$, we obtain 
\[ \frac{\Vert \rho^{-\alpha }D\phi\Vert_q^2}{(D\phi,\phi)}=\frac{\Vert \rho^{-\alpha }D\phi_i\Vert_q^2}{(D\phi_i,\phi)}=\frac{\Vert \rho^{-\alpha }D\phi_i\Vert_q^2}{(\phi_i,D\phi)}=\frac{\Vert \rho^{-\alpha }D\phi_i\Vert_q^2}{(\phi_i,D\phi_i)}.\]
Thus, for complete manifolds we can also take the infimum over all spinors $\phi\in im_{L^q} D\cap \rho^\alpha L^p \cap \rho^{-\alpha}H_1^q$.\\
\textbf{ii)} The Euler-Lagrange equation reads 
\[D(\lambda^\alpha_q \phi- \rho^{-\alpha q}|D\phi|^{q-2}D\phi)=0\mathrm{\ with\ }\Vert\rho^{-\alpha}\!D\phi\Vert_q=1.\]
As in \cite[Lem. 2.7]{Amm03a}, we obtain a dual equation by setting $\psi=\rho^{-\alpha q}|D\phi|^{q-2}D\phi$:
\bqw%\label{ELdu}
D\psi=\lambda^\alpha_q\rho^{\alpha p}|\psi|^{p-2}\psi\mathrm{\ with\ }\Vert\rho^\alpha\psi\Vert_p=1.\eqw
\textbf{iii)} Instead of considering the above invariants, we could replace the denominator by the absolute value of the former expression, i.e. $|(D\phi, \phi)|$, and allow all $\phi$ with $(D\phi,\phi)\ne 0$. For the corresponding invariant everything below works as well. These invariants might be in general smaller than the corresponding $\lambda^\alpha_q$.
\end{rem}

\begin{lem}\label{Dine}\hfill\\
\textbf{i)}\ \ \ For $\alpha\leq \beta$, $\lambda^\alpha_q \leq \lambda^\beta_q$. Moreover, $\lim_{\alpha\to 0}\lambda^\alpha_q=\lambda_q$. \\
\textbf{ii)}\ \ $\lambda^\alpha_q\geq \limsup_{s\to q}\lambda^\alpha_s$
\end{lem}
\begin{proof}\hfill\\
i) From $0< \rho\leq 1$, we obtain $\Vert \rho^{-\alpha}D\phi\Vert_q\leq \Vert \rho^{-\beta}D\phi\Vert_q$ for $\alpha\leq \beta$. Since $\lambda^\alpha_q$ is always non-negative, $\lambda_q^\alpha\leq\lambda_q^\beta$ and $\lim_{\alpha\to 0}\lambda_q^\alpha\geq \lambda_q$. Since $\lim_{\alpha\to  0} \Vert \rho^{-\alpha} D\phi\Vert_q=\Vert D\phi\Vert_q$, we obtain $\lim_{\alpha\to 0}\lambda_q^\alpha= \lambda_q$.\\
ii) Since $\Vert \rho^{-\alpha} D\phi\Vert_s\to\Vert \rho^{-\alpha}D\phi\Vert_ q$ for $s\to q$. We have $\lambda^\alpha_q\geq \limsup_{s\to q}\lambda^\alpha_s$. 
\end{proof}

Next, we show that under suitable assumptions we can ensure the positivity of $\lambda_s$:

\begin{lem}\label{pos} Assume that the Sobolev embedding $H_1^q\hookrightarrow \rho^\alpha L^p$ is continuous where $q$ and $p$ are conjugate, $2\leq p\leq p_{crit}=\frac{2n}{n-1}$ and $\alpha\geq 0$. Moreover, let there be an $L^q$-lower bound. Then, $\lambda_q^\alpha>0$.
If there is an $\epsilon >0$ such that the Sobolev embedding from above is continuous for all $q\in[q_{crit}, q_{crit}+\epsilon]$ and there is an $L^{[q_{crit}, q_{crit}+\epsilon]}$-lower bound, then $\limsup_{q\to q_{crit}} \lambda_q >0$.
\end{lem}

\begin{proof} Let $\phi\in im_{L^q} D\cap \rho^{-\alpha}H_1^q$. Then,
\begin{align*} (D\phi,\phi)&\leq \Vert \rho^\alpha\phi\Vert_p\Vert \rho^{-\alpha}D\phi\Vert_q\\%\leq \Vert \phi\Vert_p \Vert\rho^{-\alpha}D\phi\Vert_q \\
&\leq  C_{\alpha,q}(\Vert \phi\Vert_q+\Vert D\phi\Vert_q)\Vert \rho^{-\alpha}D\phi\Vert_q\leq \tilde{C_{\alpha,q}} \Vert \rho^{-\alpha}D\phi\Vert_q^2\end{align*}
where the first inequality is the H\"older inequality, the third one is given by the Sobolev embedding and the last one is obtained from the $L^q$-lower bound and $|\rho|\leq 1$. With Remark \ref{remim}.i, we obtain $\lambda_q^\alpha>0$.\\
%\m{From \cite[Proof of Lem. 3.1]{Heb}, we know that the constant $C$ in the Sobolev embedding used above can be chosen such that it is valid for all $q\in [q_{crit}, q_{crit}+\epsilon)$}. 
With the additional assumptions and since $C_{\alpha,q}$ depends continuously on $q$, we obtain analogously the second claim setting $\alpha=0$.
\end{proof}

\begin{rem}\label{posrem}%\hfill\\ \textbf{i)}\  
In particular, Lemma \ref{pos} gives a simple proof of the fact that on closed manifolds $\lm$ is always positive \cite[Lem. 4.3.1(1)]{Ammha}. With Theorem \ref{Skrzspin} we have the continuity of the above embedding on manifolds of bounded geometry. Thus, for manifolds of bounded geometry with an $L^{q_{crit}}$-lower bound $\lm$ is always positive.\\
For a general manifold, $\lm$ can be zero, e.g. if a complete spin manifold of finite volume has a positive element in the essential spectrum of its Dirac operator \cite[Lem. 3.3.iii]{NG2}.%\\
%\textbf{ii)} In general we cannot expect equality in Lemma \ref{Dine}.ii and also the converse of Lemma \ref{pos} does not hold. The standard Euclidean space $\R^n$ gives a counterexample:
%$\lm(\R^n)=\lm (S^n)>0$. But $\lambda_q(\R^n)=0$ for $q<q_{crit}$\m{?}. 
\end{rem}

We come now to the weighted subcritical problem.

\begin{lem}\label{Dwei}Let the weighted Sobolev embedding $H_1^q\hookrightarrow \rho^{\alpha} L^p$ be compact for $\alpha >0$, $2\leq p<p_{crit}=\frac{2n}{n-1}$ and let $q$ be conjugate to $p$. Furthermore, we assume that there is an $L^q$-lower bound. % and $\lambda^\alpha_q>0$. 
Then there exists a spinor $\phi\in H_1^q$ with $D\phi=\lambda^\alpha_q\rho^{\alpha p} |\phi|^{p-2}\phi$ and $\Vert\rho^\alpha \phi \Vert_p=1$. Moreover, $\phi$ is smooth outside its zero set.
\end{lem}

\begin{proof} From Lemma \ref{pos}, we obtain that $\lambda^\alpha_q >0$. Let $\phi_i$ be a normalized minimizing sequence for $\lambda^\alpha_q$, i.e. $\Vert \rho^{-\alpha} D\phi_i\Vert_q=1$ and $(D\phi_i,\phi_i)\to (\lambda^\alpha_q)^{-1}$. Due to Remark \ref{remim}.i we can assume that $\phi_i\in \rho^{-\alpha}H_1^q\cap im_{L^q} D\subset H_1^q$. Since $|\rho|\leq 1$, $\Vert D\phi_i\Vert_q\leq 1$. Due to the $L^q$-lower bound, $\phi_i$ is uniformly bounded in $H_1^q$.  Hence, $\phi_i\to\phi$ weakly in $H_1^q$ and, with the compact Sobolev embedding, even strongly in $\rho^{\alpha} L^p$. Moreover, since $D\phi_i$ is bounded in $\rho^{-\alpha}L^q$, $\rho^{-\alpha}D\phi_i$ converges weakly in $L^q$ to $\rho^{-\alpha}D\phi$. Then, $\Vert \rho^{-\alpha}D\phi\Vert_q\leq 1$ and 
\begin{align*} |(D\phi,\phi)-(D\phi_i,&\phi_i)|\leq |(\rho^{-\alpha}D\phi-\rho^{-\alpha}D\phi_i, \rho^\alpha\phi)|+|(\rho^{-\alpha}D\phi_i, \rho^\alpha\phi-\rho^\alpha\phi_i)|\\
&\leq \underbrace{|(\rho^{-\alpha}D\phi-\rho^{-\alpha}D\phi_i, \rho^\alpha\phi)|}_{\to 0\mathrm{\ since\ }  D\phi_i\xrightarrow{w-\rho^{-\alpha}L^q} D\phi}+\hspace{-0.3cm} \underbrace{\Vert \rho^\alpha\phi-\rho^\alpha\phi_i\Vert_p}_{\to 0\mathrm{\ since\ }\phi_i\xrightarrow{s-\rho^{\alpha}L^p}\phi}\hspace{-0.3cm} \underbrace{\Vert\rho^{-\alpha}D\phi_i\Vert_q.}_{=1}
\end{align*}
Hence, $(D\phi, \phi)=(\lambda^\alpha_q)^{-1}>0$ and we have \[\lambda^\alpha_q\leq \frac{\Vert \rho^{-\alpha}D\phi\Vert_q^2}{(D\phi,\phi)}\leq \lim_{i\to\infty} \frac{\Vert \rho^{-\alpha}D\phi_i\Vert_q^2}{(D\phi_i,\phi_i)}=\lambda^\alpha_q\]
which already implies that $\Vert \rho^{-\alpha}D\phi\Vert_q=1$. Thus, $\phi$ fulfills the Euler-Lagrange-equation and due to Remark \ref{remim}.ii we obtain a spinor $\psi\in H_1^q$ with $D\psi=\lambda^\alpha_q\rho^{\alpha p}|\psi|^{p-2}\psi$ with $\Vert\rho^\alpha\psi\Vert_p=1$. Smoothness outside the zero set follows from local elliptic regularity theory.
\end{proof}

%%%%%%%%%%%%%%%%%%%%%%%%%%%%%%%%%%%%%%% max and C^1-convergence
Now, we establish the first step of the convergence , we fix $\alpha$ and let $p\to p_{crit}$.

\begin{lem}\label{Dconv}Let $\phi_{\alpha,p}\in H_1^q$ be solutions of $D\phi_{\alpha, p}=\lambda^\alpha_q\rho^{\alpha p}|\phi_{\alpha, p}|^{p-2}\phi_{\alpha,p}$ with $\Vert \rho^\alpha \phi_{\alpha,p}\Vert _p=1$, $\alpha >0$ is fixed, $p\in [2,p_{crit})$ and $q$ conjugate to $p$.
Furthermore, let $M$ have bounded geometry, $\lm(S^n) > \lambda^\alpha_{q_{crit}}(M)$ and let $D$ have an $L^q$-lower bound.\\
Then, $\phi_p:=\phi_{\alpha,p}\to \phi_\alpha$ in the $C^1$-topology on each compact subset as $p\to p_{crit}$ and \[D\phi_\alpha= \lambda_{q_{crit}}^\alpha \rho^{\alpha p_{crit}} |\phi_\alpha|^{p_{crit}-2}\phi_\alpha\]
and $\Vert \rho^\alpha \phi_\alpha\Vert_{p_{crit}}=1$.
\end{lem}

\begin{proof} Note first that the parameters $p$ and $q$ are always coupled. Thus, if $q\to q_{crit}$, $p\to p_{crit}$.\\
From Lemma \ref{max}, we know that each $|\phi_p|$ has a maximum. Firstly, we show by contradiction that $m_p:=\max |\phi_p|$ is uniformly bounded. Thus, we assume $m_p\to \infty$. Let $x_p\in M$ be a point where $|\phi_p|$ attains its maximum. Around each $x_p$, we introduce geodesic normal coordinates on a ball of radius $\epsilon$ that is smaller than the injectivity radius $\mathrm{inj}(M)$, compare \cite[Lem. 11]{NG3}. We define $\delta_p= m_p^{2-p}$ and rescale $\psi_p (x)=m_p^{-1} \phi_p(\delta_p x)$. The spinor $\psi_p$ is transported to a spinor on a ball on $\R^n$ via the exponential map. For simplicity we just denote the resulting spinor by $\psi_p$, too. The weight function in the new coordinates will be denoted by $\rho_p$. Then, in the resulting rescaled geodesic coordinates the Euler-Lagrange-equation reads
\[ D^{\R^n}\psi_p+\frac{1}{4}\sum_{ijk} \tilde{\Gamma}^k_{ij}\delta_i\cdot\delta_j\cdot\delta_k\cdot\psi_p +\sum_{ij} (b_i^j-\delta_i^j)\delta_i\cdot\nabla_{\delta_j}\psi_p=\lambda_q^{\alpha}\rho^{\alpha p}_p|\psi_p|^{p-2}\psi_p\]
 where $\psi_p$ is defined on a ball in $\R^n$ of radius $\frac{\epsilon}{\delta_p}$  and
\begin{align}\label{Dcoeff}
b_i^j(x)&=\delta_i^j-\frac{1}{6}\delta_p^2  R_{i\gamma\beta j}(\delta_p x)x^\gamma x^\beta +O(\delta_p^2|x|^2)\to\delta_i^j\nonumber\\[-0.3cm]
&\\[-0.3cm]
\tilde{\Gamma}_{ij}^k(x)&=\partial_i b_j^k -\frac{1}{3} \delta_p(R_{ik\gamma j}(\delta_p x)+R_{i\gamma kj}(\delta_p x))x^\gamma+O(\delta_p^2 |x|^2)\to 0\nonumber
\end{align}
as $\delta_p\to 0$ (i.e. as $p\to p_{crit}$).\\
Here we used the local comparison of a Dirac operator with the one of the standard Euclidean space \cite[Sect. 3 and 4]{Amm03b}.\\
Since the Riemannian curvature and its derivatives are bounded, we have $C^1$-convergence of \eqref{Dcoeff} on any compact subset of $\R^n$. Thus, for any compact subset $K$ and an open set $\Omega$ with $K\subset \Omega \subset \R^n$ we obtain with interior Schauder and $L^p$-estimates (see \cite[Sect. 3.1.5 and 3.2.2]{Ammha}) that $\psi_p\to \psi$ in $C^1$ on $K$.  Thus, we have a spinor $\psi$ on $\R^n$ with
$D^{\R^n}\psi= (\lim_{p\to p_{crit}} \rho^{\alpha p_{crit}}_p\lambda_q^{\alpha}) |\psi|^{p_{crit}-2}\psi$. Moreover, $|\psi|\leq 1$ with $|\psi(0)|=1$. As in the Yamabe problem, we prove that $\Vert \psi\Vert_{p_{crit}, g_E}\leq 1$: For $p< p_{crit} $ we have 
\begin{align*}
\int_{|x|<\epsilon \delta^{-1}_p} |\psi_p|^{p} b_p\,\vo_{g_E}&=\int_{B_\epsilon(x_p)} |\phi_p|^{p} \delta_p^{\frac{2p}{p-2}-n}\,\vo_g\\
&\leq   C(\epsilon)\Vert \phi_p\Vert^p_{H_1^q}\leq \tilde{C}(\epsilon) \Vert D\phi_p\Vert^p_q=\tilde{C}(\epsilon)(\lambda^\alpha_q)^p
\end{align*}
where $b_p\vo_{g_E}$ denotes the transported volume element with $b_p\to 1$ as $p\to p_{crit}$, we used the Sobolev embedding on $B_\epsilon (x_p)$ (its constant only depends on $\epsilon$ since $M$ is of bounded geometry) and we used the $L^q$-lower bound. Thus, with Fatous Lemma we obtain $\psi\in L^{p_{crit}}(g_E)$.

Consider firstly the case that $x_p$ espape to infity as $p\to p_{crit}$. Then $\rho_p\to 0$ and, hence, $D^{\R^n}\psi=0$. Thus, we have an $L^{p_{crit}}$-harmonic spinor on $\R^n$ with $|\psi(0)|=1$ which is a contradiction. Thus, we can stick to the case that $x_p\to y\in M$. Then, for $\epsilon_1<\epsilon$
\begin{align*}
\int_{|x|<\epsilon_1\delta^{-1}_p}|\psi_p|^{p} b_p\,\vo_{g_E}&\leq \max_{B_{\epsilon_1}(x_p)} \rho^{-\alpha p}\int_{B_{\epsilon_1}(x_p)} \rho^{\alpha p} |\phi_p|^{p} \delta_p^{\frac{2p}{p-2}-n}\,\vo_g\\
&\leq  \delta_p^{\frac{2p}{p-2}-n} \max_{B_{\epsilon_1}(x_p)} \rho^{-\alpha p}\leq \max_{B_{\epsilon_1}(x_p)} \rho^{-\alpha p}.
\end{align*}
With Fatous Lemma we obtain $\Vert \psi\Vert_{p_{crit}, g_E}\leq \max_{B_{\epsilon_1}(y)} \rho^{-\alpha }$. Since $\epsilon_1$ can be chosen arbitrarily small we have $\Vert \psi\Vert_{p_{crit}, g_E}\leq \rho^{-\alpha}(y)$. Thus, 
\begin{align*} \lm(S^n)&=\lm(\R^n)\leq \frac{\Vert D^{\R^n}\psi\Vert_{q_{crit}}^2}{(D^{\R^n}\psi,\psi)}\\ &=(\lambda_{q_{crit}}^\alpha \rho^{\alpha p_{crit}\left( \frac{2}{q_{crit}}-1\right)}(y) \Vert \psi\Vert_{p_{crit}}^{p_{crit}(\frac{2}{q_{crit}}-1)}\leq \lambda_{q_{crit}}^\alpha
 \end{align*}
which gives a contradiction. Thus, $|\phi_\alpha|$ is bounded from above. Using again interior Schauder and $L^p$-estimates, we can then show that $\phi_\alpha$ is uniformly bounded in $C^{1,\alpha}$ on each bounded subset and, hence, converges in $C^1$ on every compact subset $K\subset M$  to a spinor $\phi_\alpha$. Thus, we have $D\phi_\alpha\leq \lambda^\alpha_{q_{crit}} \rho^{\alpha p_{crit}}|\phi_\alpha|^{p_{crit}-2}\phi_\alpha$ and $\Vert \rho^{\alpha}\phi_\alpha\Vert_{p_{crit}}\leq 1$.\\
We still need to show that $\Vert \rho^{\alpha}\phi_\alpha\Vert_{p_{crit}}=1$:\\
Assume that $\phi_\alpha=0$. Fix $R>0$ and let $B_R$ be the ball of radius $R$ around the fixed $z\in M$ (cf. Definition of $\rho$ in Section \ref{pr}) we obtain  that 
\[ \int_{M\setminus B_R} |D\phi_p|^q\vo_g \leq \max_{M\setminus B_R} \rho^{\alpha q} \Vert \rho^{-\alpha}D\phi_p\Vert_q^q \leq C e^{-R\alpha q}(\lambda^\alpha_q)^q.\] With 
\[ \lambda_q \leq \frac{\Vert D\phi_p\Vert_q^2}{(D\phi_p,\phi_p)}= \left( \lambda^\alpha_q\right)^{-1} \Vert D\phi_p\Vert_q^2\leq \left( \lambda^\alpha_q\right)^{-1} \Vert \rho^{-\alpha} D\phi_p\Vert_q^2=\lambda^\alpha_q,\]
we then get \begin{align*}
 \int_{B_R} |D\phi_p|^q\vo_g &= \int_{M} |D\phi_p|^q\vo_g-\int_{M\setminus B_R} |D\phi_p|^q\vo_g \\
 &\geq (\lambda^\alpha_q\lambda_q)^{\frac{q}{2}}- Ce^{-R\alpha q}(\lambda^\alpha_q)^q\end{align*}
and ,thus,
\[ \limsup_{q\to q_{crit}} \int_{B_R} |D\phi_p|^q\vo_g\geq (\lambda^\alpha_{q_{crit}}\limsup_{q\to q_{crit}} \lambda_q)^{\frac{q_{crit}}{2}}- Ce^{-R\alpha q_{crit}}(\lambda^\alpha_{q_{crit}})^{q_{crit}}=:a\]
From Lemma \ref{pos} we know that $\limsup_{q\to q_{crit}} \lambda_q >0$. Thus, $R=R(\alpha)$ can be chosen big enough, that $a>0$. But with \[\int_{B_R} |D\phi_p|^q \vo_g\leq \left(\int_{B_R} |D\phi_p|^{q_{crit}} \vo_g\right)^{\frac{q}{q_{crit}}}\vol(B_R,g)^{1-\frac{q}{q_{crit}}}\]
 this contradicts the assumption that $\phi_p$ goes to zero on compact subsets. Thus, $\phi_\alpha\ne 0$ and we obtain with \[0<\lambda^\alpha_{q_{crit}}\leq \frac{\Vert D\phi_\alpha\Vert_{q_{crit}}^2}{(D\phi_\alpha,\phi_\alpha)}\leq \lambda^\alpha_{q_{crit}} \Vert \rho^\alpha \phi_\alpha\Vert_{p_{crit}}^{2-q_{crit}}\] 
in the usual way that $\Vert \rho^\alpha \phi_\alpha\Vert_{p_{crit}}=1$ and $D\phi_\alpha= \lambda_{q_{crit}}^\alpha \rho^{\alpha p_{crit}} |\phi_\alpha|^{p_{crit}-2}\phi_\alpha$.
\end{proof}

Lastly, we get rid of the weight as $\alpha\to 0$.

\begin{lem}\label{Dinfmass}
Assume there is a sequence of smooth and positive spinors $\phi_{\alpha}\in H_1^{q_{crit}}$ fulfilling $D\phi_{\alpha}=\lambda^\alpha_{q_{crit}}\rho^{\alpha p_{crit}}|\phi_{\alpha}|^{p_{crit}-2}\phi_{\alpha}$ with $\Vert \rho^\alpha \phi_{\alpha}\Vert_{p_{crit}}=1$ and $\alpha>0$. Then, $\phi_\alpha\to \phi$ as $\alpha\to 0$ in $C^1$ on each compact subset and $D\phi\leq \lm |\phi|^{p_{crit}-2}\phi$.\\
If additionally $\ov{\lm(M)}>\lm(M)>0 $, then $\Vert \phi\Vert_{p_{crit}}=1$.
\end{lem}

\begin{proof} The first claim is proven as in Lemma \ref{Dconv} and we obtain that $\phi_{\alpha} \to \phi$ as $\alpha\to 0$ in $C^1$ on each compact subset with $D\phi\leq \lm|\phi|^{p-2}\phi$.\\
Let now $\ov{\lm(M)}>\lm(M)$. We need to show that $\Vert \phi\Vert_{p_{crit}}=1$.\\
Assume that $\phi_{\alpha}$ converges to $0$: \\
We introduce a smooth cut-off function $\eta_R\leq 1$ with support in $M\setminus B_R(z)$ for the fixed $z\in M$, $\eta_R\equiv 1$ on $M\setminus B_{2R}$. Then,
\begin{align*}
&\left|\int_{M\setminus B_R}\!\!\ |D (\eta_R \phi_{\alpha})|^{q_{crit}}\vo_g-\int_{M\setminus B_R}\!\! |D \phi_{\alpha}|^{q_{crit}}\vo_g\right|\to 0  \quad \mathrm{for}\ \alpha\to 0
\end{align*}
since $\phi_{\alpha}\to 0$ in $C^1$ on each compact set. Hence, with $(\phi_{\alpha}, D\phi_{\alpha})=\lambda^\alpha_{q_{crit}}$  we get 
\begin{align*}
\lm(M)^2 &= \lim_{\alpha\to 0 }\lambda_{q_{crit}}^{\alpha}(M)^2= \lim_{\alpha\to 0} \Vert\rho^{-\alpha} D \phi_\alpha\Vert_{q_{crit}}^2\\
& \geq \lim_{\alpha\to 0} \Vert D \phi_\alpha\Vert_{q_{crit}}^2=\lim_{\alpha\to 0} \Vert D (\eta_R \phi_\alpha)\Vert_{q_{crit}}^2\\
&\geq \lm (M\setminus  B_R) \lim_{\alpha\to 0}  \int_{M\setminus B_R}\!\!\!\langle \eta_R \phi_\alpha, D (\eta_R \phi_\alpha)\rangle \vo_g\\
&= \lm (M\setminus B_R) \lim_{\alpha\to 0}  \left( (\phi_\alpha, D\phi_\alpha)- \int_{B_{2R}}\!\!\!(1-\eta_R^2)\langle \phi_\alpha, D\phi_\alpha\rangle \vo_g\right)\\
&=\lm (M\setminus B_R) \lim_{\alpha\to 0} \lambda_{q_{crit}}^{\alpha}\left( 1- \int_{B_{2R}}\!\!\!(1-\eta_R^2) \rho^{\alpha p_{crit}} |\phi_\alpha|^{p_{crit}}\vo_g\right)\\
&=\lm (M\setminus B_R) \lm(M)\end{align*} 
where the first and last equality use Lemma \ref{Dine}.i. Since $\lm >0$, this implies $\lm(M)\geq \lm(M\setminus B_R)$ which gives the contradiction.\\

Thus, $\Vert \phi\Vert_{p_{crit}}>0$. With $D\phi\leq \lm |\phi|^{p_{crit}-2}\phi$ and $\Vert\phi\Vert_{p_{crit}}\leq 1$, we obtain
\[ \lm\leq \frac{\Vert D\phi\Vert_{p_{crit}}^2}{(\phi,D\phi)}\leq \lm\Vert \phi\Vert_{p_{crit}}^{\frac{2}{n-1}}\leq \lm.\]
The equality implies $\Vert\phi\Vert_{p_{crit}}=1$ and $D\phi=\lm |\phi|^{p_{crit}-2}\phi$.
\end{proof}

Combining the Lemmas above, we obtain:

\begin{proof}[Proof of Theorem \ref{main1}]
Firstly, we note that from $\ov{\lm(M)}>\lm(M)$ we obtain $\lm(S^n)>\lm(M)$. Thus, with Lemma \ref{Dine}.i there is an $\alpha_0 >0$ such that $\lm(S^n)>\lambda_{q_{crit}}^\alpha(M)$ for all $\alpha\in(0,\alpha_0)$. Moreover, Lemma \ref{pos} states that $\lm(M)>0$. Hence, we can combine Lemma \ref{Dwei}, \ref{Dconv} and \ref{Dinfmass} to obtain Theorem \ref{main1}. The local $C^{1,\alpha}$-regularity of the solution follows as in the compact case \cite[Sect. 4.4]{Ammha} by the standard backtracking argument and $L^p$- and Schauder estimates (cf. Remark \ref{est}.iii).
\end{proof}

As in the case of the Yamabe invariant \cite[Thm. 13]{NG3}, in order to drop the condition on the $\lm$-invariant at infinity we can consider homogeneous manifolds with positive scalar curvature:

\begin{thm}\label{main2} Let $(M,g,\sigma)$ be a Riemannian spin manifold. Furthermore, we assume that its Dirac operator has an $L^{[q_{crit},q_{crit}+\epsilon]}$-lower bound for an $\epsilon>0$, $\s\geq c>0$  and that there exists a relatively compact set $U\subset\subset M$ such that for all $x\in M$ there is an isometry $f:M\to M$ with $f(x)\in U$. Then there is a spinor $\phi\in H_1^{q_{crit}}\cap L^{p_{crit}}$ that is smooth outside its zero set and fulfills $D\phi=\lm |\phi|^{p_{crit}-2}\phi$ and $\Vert \phi\Vert_{p_{crit}}=1$. Moreover, $\phi$ is locally in $C^{1,\alpha}$.
\end{thm}

\begin{proof}
Due to the isometries, $M$ has bounded geometry. Thus, the Sobolev embeddings $H_1^q\to \rho^\alpha L^p$ are compact for $q\in (q_{crit},2]$, $q$ conjugate to $p$ and $\alpha>0$. Hence, we can apply Lemma \ref{Dwei} and obtain solutions $\phi_{\alpha,p}\in H_1^q$ of $ D\phi_{\alpha,p}=\lambda^\alpha_q\rho^{\alpha p}|\phi_{\alpha,p}|^{p-2}\phi_p$ and $\Vert \rho^\alpha \phi_{\alpha,p}\Vert_p=1$. \\
We choose $\alpha(p)$ such that $\lambda_q^{\alpha(p)}-\lambda_q\to 0$ and $\alpha(p)\to 0$ as $q\to q_{crit}$. Due to Lemma \ref{Dine}.i this is always possible. In the following, we abbreviate $\alpha=\alpha(p)$.\\
As in the proof of Lemma \ref{Dconv} we obtain that a subsequence of $\phi_p:= \phi_{\alpha(p), p}$ converges to $\phi$ in $C^1$-topology.\\
Lemma \ref{max} shows that every $|\phi_p|$ has a maximum. Due to the isometries we can assume that $\max |\phi_p|=|\phi_p (x_p)|$ where $x_p\in U$.
Then, we have \[D^2\phi_p=(\lambda^\alpha_q)^2\rho^{2\alpha p}|\phi_{p}|^{2(p-2)}\phi_p+\lambda^\alpha_q\,\d (\rho^{\alpha p}|\phi_{\alpha,p}|^{p-2})\cdot \phi_p.\] Taking the real part of the scalar product with $\phi_p$ we obtain using $\langle df\cdot \phi, \phi\rangle \in \mathrm{i}\R$:
\[ (\lambda^{\alpha}_q)^2\rho^{2{\alpha} p}|\phi_{p}|^{2(p-1)}=\Re \langle D^2\phi_p, \phi_p\rangle =\Re \langle \phi_p, \Delta \phi_p\rangle +\frac{\s}{4}|\phi_p|^2.\]
With $-\frac{1}{2}\Delta |\phi_p|^2+\langle \phi_p, \Delta \phi_p\rangle =|\Delta \phi_p|^2$ and $\Delta |\phi_p(x_p)|^2\geq 0$ we get
\[ (\lambda^{\alpha}_q)^2\rho^{2{\alpha} p}(x_p)|\phi_{p}(x_p)|^{2(p-1)} \geq \frac{\s(x_p)}{4}|\phi_p(x_p)|^2\geq \frac{c}{4}|\phi_p(x_p)|^2.\]
Since $\rho\leq 1$,
\[ |\phi_{p}(x_p)|^{2(p-2)} \geq \frac{c}{4}(\lambda^{\alpha}_q)^{-2}.\]
Thus, with $\lim_{p\to p_{crit}} \lambda_q^{\alpha}=\lim_{p\to p_{crit}} \lambda_q\geq \lm$ and $ \lm >0$, cf. Lemma \ref{pos}, we get
\bqw |\phi(x)|^{2(p_{crit}-2)} \geq \frac{c}{4} (\lm)^{-2}\eqw  
where $x\in \ov{U}$ is the limit of a convergent subsequence of $x_p$. Hence, $\Vert \phi\Vert_{p_{crit}}>0$ and, thus, as in Lemma \ref{Dinfmass} we have $D\phi=\lm|\phi|^{p_{crit}-2}\phi$ with $\Vert \phi\Vert_{p_{crit}}=1.$ Local $C^{1,\alpha}$-regularity follows as in Theorem \ref{main1}.
\end{proof}

\begin{ex} Let $M=S^n\times \R$ for $n\geq 2$ be equipped with the standard product structure and its unique spin structure. Then, $M$ is homogenous and has positive scalar curvature. In order to apply Theorem \ref{main2} one has to show that $M$ has an $L^{[q_{crit}, q_{crit}+\epsilon]}$-lower bound for a small $\epsilon>0$. Assume that $M$ has an $L^q$-harmonic spinor $\phi$ for $q\in [q_{crit}, q_{crit}+\epsilon]$ then due to the Sobolev embedding gives that $\phi$ is already an $L^2$-harmonic spinor which cannot happen since $M$ has positive scalar curvature. Thus, it remains to consider the essential $L^q$-spectrum. It can be checked that the Dirac operator in this example is $L^q$ invertible for all q and, thus, a minimizing solution is obtained.     
\end{ex}

In the remark below we want to examine open manifold which are spin conformally compactifiable by codimension greater or equal two. We want to study how the solutions on the manifold and its compactification are related.
 
\begin{rem}
Let $(M,g,\sigma)$ be an open connected Riemannian spin manifold that is spin conformally compactifiable to a closed Riemannian spin manifold $(N,h,\xi)$, i.e. there a conformal embedding $f: M\to N$ that is compatible with the spin structures. Assume that $S:=N\setminus f(M)$ is a compact submanifold of $N$ of codimension $\geq 2$. Then, $\lm (N)=\lm (M)=:\lambda$ (see \cite[Lem. 2.1]{NG} for codimension $>2$ and \cite[Lem. 3.1.1]{NGDiss}). On $N$ the Euler-Lagrange equation has at least one solution $\phi\in H_1^{q_{crit}}$ of
\[ D\phi= \lambda |\phi|^{p_{crit}-2}\phi\quad \Vert \phi\Vert_{p_{crit}}=1\]
that is $C^{1,\alpha}$ and smooth outside its zero set. Since the Euler-Langrange equation including the $L^{p_{crit}}$-condition is conformally invariant, each solution $\phi$ on $N$ gives a solution $f^*(\phi_{|_{f(M)}})$ on $M$.

On the other hand, let $\psi$ be a solution on $N$. Then, $\phi=f_*(\psi)$ is a solution on $f(M)$ just by conformal invariance. Due to \cite[Thm. 3.1]{Ammtau} we can remove the singularity in $S$ and $\phi$ fulfills the Euler-Lagrange equation on $M$. This showed that we have a one-to-one of solutions on $M$ and such on $N$. 
\end{rem}

\section{Conformal Hijazi inequality} \label{cohij}

On closed manifolds, the Hijazi inequality gives an estimate for eigenvalues of the Dirac operator in terms of the lowest eigenvalue of the conformal Laplacian $L=4\frac{n-1}{n-2}\Delta +\s$:
\begin{thm} \cite{Hija}
 Let $(M^n,g,\sigma)$ be a closed spin manifold of dimension $n$. Let further $\lambda$ be an eigenvalue of the Dirac operator and $\mu$ the lowest eigenvalue of $L$. Then
\[ \lambda^2\geq \frac{n}{4(n-1)} \mu.\]
\end{thm}

The conformal Laplacian is highly related to the Yamabe invariant $Q$: For closed manifolds with $Q\geq 0$ we have
\[ Q(M,g)=\inf_{\ov{g}\in [g]} \mu(\ov{g})\vol(M,\ov{g})^\frac{2}{n},\]
where $[g]$ denotes the conformal class of $g$. An analogous result holds for $\lm$ on closed manifolds \cite[Prop. 2.6]{Amm03a}:
\[ \lm(M,g, \sigma)=\inf_{\ov{g}\in [g]} \lambda_1^+(\ov{g})\vol(M,\ov{g})^\frac{1}{n},\]
where $\lambda_1^+(\ov{g})$ is the first positive eigenvalue of the Dirac operator w.r.t. the metric $\ov{g}$. Thus, there is a conformal version of the Hijazi inequality on closed manifolds:
\[ \lm(M,g, \sigma)^2\geq \frac{n}{4(n-1)} Q(M,g).\]

Note that only the left side of the inequality depends on the spin structure.\\

In \cite{NG2}, we examined whether the Hijazi inequality on complete open manifolds holds when one replaces the eigenvalues by elements of the corresponding spectra. There we saw that for general open manifolds there is no hope for an Hijazi inequality, the standard hyperbolic space gives a counterexample \cite[Rem. 4.2.1.]{NGDiss} since $\mu(\Hy^n)= \frac{(n-1)^2}{4}>0$ but $\lambda_1^+(\Hy^n) = 0$. That is why we restricted in \cite{NG2} to complete manifolds of finite volume in order to obtain an Hijazi inequality on those manifolds. 

But the conformal Hijazi inequality is e.g. also valid for the hyperbolic space since $Q(\Hy^n)=Q(S^n)$ and $\lm(\Hy^n)=\lm(S^n)$ since $\Hy^n$ is conformally equivalent to a subset of the standard sphere. Thus, there is hope for the conformal Hijazi inequality to hold more generally.\\

Below we will show that if there is a solution of the Euler-Lagrange equation, the conformal version of the Hijazi inequality holds: 

\begin{thm}\label{Hij-lam}
Let $(M,g, \sigma)$ be a complete $n$-dimensional Riemannian spin manifold  with $n\geq 3$. Moreover, assume that there exists a spinor $\phi\in H_1^q(M,S)\cap L^p(M,S)$, smooth outside its zero set, with $p=\frac{2n}{n-1}$ and $q$ conjugate to $p$ such that it is a weak solution of $D\phi=\lambda |\phi|^{p-2}\phi$ and $\Vert\phi\Vert_p=1$. Then, 
\[ \lambda^2\geq \frac{n}{4(n-1)}Q.\]
In particular, if $\phi\in ker_{L^p}D\cap H_1^q$, then $Q\leq 0$.
\end{thm}

\begin{proof}
We define $\alpha=\frac{n-2}{n-1}$. We note that due to the unique continuation property of the Dirac operator the zero set of $\phi$ has zero $n$-dimensional volume.
\begin{align*}
\frac{Q}{4}&\left( \int_M |\phi|^{\alpha \frac{2n}{n-2}}\vo_g\right)^\frac{n-2}{2} -\frac{n-1}{n}\lambda^2 \int_M |\phi|^p\vo_g \\
&\leq \frac{1}{4}\int_M |\phi|^\alpha L |\phi|^\alpha\vo_g -\frac{n-1}{n}\lambda^2 \int_M |\phi|^{2(p-2)}|\phi|^{2\alpha}\vo_g\\
&\leq\int_M |\phi|^{2\alpha-2} \left( \frac{n}{n-1}|\d |\phi||^2+\frac{1}{2}\d^*\d |\phi|^2\right.\\
&\quad\quad\left.+\left(\frac{\s}{4}-\frac{n-1}{n}\lambda^2|\phi|^{2(p-2)}\right)|\phi|^2\right)\vo_g
\end{align*}

Using $\frac{1}{2}\d^*\d|\phi|^2=\langle D^2\phi,\phi\rangle -\frac{\s}{4}|\phi|^2-|\nabla\phi|^2$ and $|\nabla\phi|^2=|\nabla^f\phi|^2+2\frac{f}{n}\langle (D-f)\phi,\phi\rangle+\frac{f^2}{n}|\phi|^2$ with $f=\lambda|\phi|^{2(p-2)}$, we obtain
\begin{align*}
\frac{Q}{4} -\frac{n-1}{n}\lambda^2&\leq \int_M |\phi|^{2\alpha-2} \left(\frac{n}{n-1}|\d|\phi||^2-|\nabla^f\phi|^2+\langle (D^2-f^2)\phi,\phi\rangle \right)\vo_g\\
&=\int_M |\phi|^{2\alpha-2}\left(\frac{n}{n-1}|\d|\phi||^2-|\nabla^{\lambda |\phi|^{p-2}}\phi|^2\right)\vo_g\leq 0\\
\end{align*}

where we used the refined Kato inequality \cite[(3.9)]{CGH}.

These estimates imply that $\int_M |\phi|^\alpha L |\phi|^\alpha\vo_g < \infty$.
\end{proof}

Theorem \ref{Hijsum} is now just a combination of the Theorems \ref{main1}, \ref{main2} and \ref{Hij-lam}.

%%%%%%%%%%%%%%%%%%%%%%%%%%%%%%%%%%%%%%%%%%%%%%%%%%%%%
%%%%%%%%%%%%%%%%%%%%%%%%%%%%%%%%%%%%%%%%%%%%%%%%%%%%%

\begin{appendix}
\section{Sobolev embeddings}\label{app}

Let $E$ be a hermitian vector bundle on a Riemannian manifold $(M,g)$ with connection $\nabla$. Denote by $\Gamma_c (E)$ the compactly supported smooth sections of $E$. Define for $u\in \Gamma_c (E)$
\bq\label{normE}\Vert u\ |{H_s^p(E)}\Vert^p=\sum_{j=0}^s \int_M |\underbrace{\nabla\cdots\nabla}_{j\ times} u|^p\vo_g\eq
and $H_s^p(E)$ as the completion of $\Gamma_c (E)$ with respect to this norm.\\

On closed manifolds $M^n$, the Sobolev spaces do not depend on the metric on $M$ and the connection on $E$. Sobolev embeddings on hermitian vector bundles can then be examined by the following procedure: One uses the known Sobolev embeddings for complex-valued functions on $\R^n$ and generalize them to $\C^r$-valued functions on $\R^n$. This generalization follows immediately since $\Vert f=(f_1,\ldots, f_r)^T\ | H_s^p(\R^n, \C^r)\Vert \sim \sum_{i=0}^r \Vert f_i\ | H_s^p(\R^n)\Vert$, cf. Lemma \ref{trback}. Then, one uses a finite covering of $M$ with a corresponding trivialization of $E$ to get a Sobolev embedding for $E$.\\

On open manifolds, in general there is a dependence on both the metric and the connection. We will concentrate on bundles of bounded geometry, i.e. the manifold itself is of bounded geometry and the curvature of $E$ and all its derivatives are bounded as well.\\

In \cite{Skrz}, continuity and compactness for weighted Sobolev embeddings were studied for Sobolev spaces of real valued functions. We will need the same result for Sobolev spaces of bundles, especially the compactness of the embedding $H_1^q(M,S) \hookrightarrow \rho L^p(M,S)$ on spinors for $q$ and $p$ conjugate, $\rho$ a radial weight (see \ref{weight}) and $2\leq p< p_{crit}=\frac{2n}{n-1}$. In this Appendix, we want to use the result from Skrzypczak \cite{Skrz} to obtain the following theorem (or more generally a bundle version of Theorem \ref{mainSkrz}):

\begin{thm}\label{Skrzspin}
Let $(M,g)$ be an $n$-dimensional manifold with an hermitian vector bundle $E$ of bounded geometry (e.g. if $M$ is a spin manifold, let $E$ be the corresponding spinor bundle). Let $\rho$ be a radial weight with $\rho\to 0$ as $|x|\to\infty$. Then the Sobolev embedding $H_1^q(E)\hookrightarrow L^p(E)$ is continuous and $H_1^q(E)\hookrightarrow \rho L^p(E)$ is compact for all $2\leq p < p_{crit}=\frac{2n}{n-1}$ and $\frac{1}{n}$ and  $\frac{1}{n}> \frac{1}{q}-\frac{1}{p}$.
\end{thm}

If one is just interested in continuous embedding, it is much easier to carry over the local results to manifolds of bounded geometry by chosing a suitable covering as it is indicated in Corollary \ref{cont_emb}.

\subsection{Weighted function spaces}\label{weight}

\begin{defi}\ \!\cite[Def. 2+4]{Skrz}
A function $w: M^n\to (0.\infty)$ is called an {\em admissible weight} if\\
i)\ \ $w$ is infinitely differentiable and\\
ii) the quantities 
\[ C_w:=\sup_{x\in M}\sup_{y\in \ov{B_1(x)}} \frac{w(y)}{w(x)}, \quad 
c_{w,k}:=\sup_{x\in M} \frac{|\nabla^k w(x)|}{w(x)}
\]
are finite for all $k\in \N_0$.\\
If additionally there is a positive function $\kappa: [0,\infty)\to (0,\infty)$ and a point $x_0\in M$ such that $w(x)\sim \kappa ( dist(x_0,x))$, the weight is called {\it radial w.r.t. $x_0$}.
\end{defi}

\begin{rem}\label{exwe}\hfill\\
\textbf{i)}\ The equivalence relation $a\sim b$ means that there exists a constant $c$  that is independent on the context dependent relevant parameters such that $c^{-1}a\leq b\leq ca$.\\ 
\textbf{ii)} An example of a radial weight is given by $w(x)=e^{\alpha \tilde{r}}$ with $\alpha\in \R$ and $\tilde{r}\sim r=dist(x,x_0)$ is a smooth substitute for the distance function \cite[Lem. 2.1]{Shu}.
\end{rem}

The weighted spaces $L^p(E,w)=wL^p(E)$ are simply defined as completion of $\Gamma_c(E)$ by the quasi-norm $\Vert \phi\ |L^p(E,w)\Vert:= \Vert w\phi\ |L^p(E)\Vert$. Analog definitions give weighted $H_p^s$-spaces and weighted Besov spaces $B_{pq}^s$. For the definition of (unweighted) Besov spaces see \cite[Chap. 2]{TFS2} or Def. \ref{beso1} and Section \ref{manbgeo}.\\

For Sobolev embeddings of spaces of scalar-valued functions, there is the following result of Skrzypczak:

\begin{thm}\ \!\cite[Thm. 2 + Cor. 2]{Skrz}\label{mainSkrz}
Let $1\leq p_1,p_2,q_1,q_2\leq \infty $ and $s_1,s_2\in \R$. Let $w(x)=v(|x|)$ be a radial admissible weight on $M^n$.\\
\textbf{i)} Then, the embedding $B_{p_1,q_1}^{s_1}(M^n,w)\hookrightarrow B_{p_2,q_2}^{s_2}(M^n)$ holds if and only if
\bq \sum_{m=1}^\infty v(m)^{-p*}a(m)\textrm{\ if\ } p*<\infty \textrm{\ or\ }\sup_mv(m)^{-1}<\infty\textrm{\ if\ }p*=\infty \label{bed1}
\eq
and
\bq s_1-s_2-n\left( \frac{1}{p_1}-\frac{1}{p_2}\right)\left\{ \begin{matrix}
                                                        \geq \frac{n}{p*}\textrm{\ if\ } q*=\infty \\ >\frac{n}{p*}\textrm{\ if\ } q*<\infty
                                                       \end{matrix} \right.\label{bed2}
                                                       \eq
where $\frac{1}{p*}:=\left(\frac{1}{p_1}-\frac{1}{p_2}\right)_+$ and $\frac{1}{q*}:=\left(\frac{1}{q_1}-\frac{1}{q_2}\right)_+$ with $a_+=\left\{ \begin{matrix} 0\textrm{\ for\ }a\leq 0\\ a\textrm{\ for\ }a> 0\end{matrix} \right.$.\\
\textbf{ii)} The embeddings $B_{p_1,q_1}^{s_1}(M^n,w)\hookrightarrow B_{p_2,q_2}^{s_2}(M^n)$ and $H_{p_1}^{s_1}(M,w)\hookrightarrow H_{p_2}^{s_2}(M)$ are compact if and only if
\eqref{bed1}, \eqref{bed2},
\[  s_1-s_2-n\left( \frac{1}{p_1}-\frac{1}{p_2}\right)>\frac{n}{p*}\textrm{\ if\ }q*=\infty\]
and 
\[ \lim_{m\to\infty}v(m)=\infty \textrm{\ if\ }p*=\infty\]
hold.
\end{thm}

{\bf Strategy of proof.}
\begin{enumerate}
 \item Define a wavelet frame $\{\psi_i\}$ on $M$.
\item  Define weighted sequence spaces $b_{p,q}^s(M,\rho)$ such that the map 
\[ B_{p,q}^s(M,\rho) \to  b_{p,q}^s(M,\rho), \quad  f\mapsto \{\lambda_i=\langle f,\psi_i\rangle \}\] is an homeomorphism. Thus, it is sufficient to prove continuity and compactness on the level of the sequence spaces, where corresponding results are known.
\end{enumerate}

Theorem \ref{mainSkrz} also holds for a vector bundle $E$ of bounded geometry. With some few adaptions which will be discussed in the following the proof can be simply overtaken. \\
For that we firstly consider trivial $\C^r$-bundles over $\R^n$, see Section \ref{triv}, where we trace back the (quasi-)norms of function spaces of vector-valued functions to the ones of scalar-valued functions. Next, we give the definition of a wavelet frame on spaces of manifolds of bounded geometry, see Sect. \ref{manbgeo}. Then, we will see that the appearing weighted sequences spaces differ from the ones in the scalar-valued case just by a finite summation that does not affect the old proof.

\subsection{Function spaces on $\R^n$ with values in a trivial hermitian bundle}\label{triv}

We start with the case where $E$ is a trivial $\C^r$-bundle over the Euclidean space $\R^n$.\\
Let the Schwarz space of all complex-valued rapidly decreasing, infinitely differentiable functions on $\R^n$ be denoted by $\mathcal{S}(\R^n)$. If $\phi\in \mathcal{S}(\R^n,\C^r)$, i.e. $\phi=(\phi_1,\ldots ,\phi_r)^T$ with $\phi_i\in \mathcal{S}(\R^n)$, then we denote the Fourier transform by
\[ \hat{\phi}(\xi)=(F\phi)(\xi)=(2\pi)^{-\frac{n}{2}}\left( \int_{\R^n} e^{-i\xi x}\phi_i(x)\d x\right)^T\!\!, \quad \xi\in \R^n.\]
The inverse Fourier transform will be denoted by $F^{-1}\phi$ or $\check{\phi}$.\\
Moreover, let $\alpha_j$ be a dyadic resolution of unity defined by: $\alpha_0\in \mathcal{S}(\R^n)$ with $\alpha_0(x)=1$ if $|x|\leq1$ and $\alpha_0(x)=0$ if $|x|\geq \frac{3}{2}$. Let $\alpha_1(x)=\alpha_0(\frac{x}{2})-\alpha_0(x)$ and $\alpha_k(x)=\alpha_1(2^{-k+1}x)$ for $x\in\R^n$ and $k\in\N$. $\alpha_0$ is chosen such that $\sum_{j=0}^\infty \alpha_j(x)=1$ on $\R^n$.

\begin{defi}\ \!\cite[Sect. 2.3.1]{TriIT}\label{beso1}
Let $s\in \R$ and $0< q\leq \infty$.\\
%i) 
Let $0<p\leq \infty$. The Besov space $B^s_{pq}(\R^n, \C^r)$ is the collection of all $f\in S'(\R^n)$ such that the norm
\[ \Vert \phi| B^s_{pq}(\R^n, \C^r)\Vert=\left(\sum_{j=0}^\infty 2^{jsq}\Vert (\alpha_j\hat{\phi})\check{}\ |\ L^p(\R^n,\C^r)\Vert^q\right)^\frac{1}{q} \]
is finite.\\
For $q=\infty$, we set $\Vert \phi \ |\ B^s_{p\infty}(\R^n, \C^r)\Vert=\sup_j 2^{sj}\Vert  (\alpha_j\hat{\phi})\check{}\ |\ L^p(\R^n,\C^r)\Vert$.
%\\ ii) Let $0<p< \infty$. The Besov space $F^s_{pq}(\R^n, \C^r)$ is the collection of all $f\in S'(\R^n)$ such that the norm
%\[ \Vert \phi\ |\ F^s_{pq}(\R^n, \C^r)\Vert_\alpha=\left\Vert \left(\sum_{j=0}^\infty 2^{jsq}| (\alpha_j\hat{\phi})\check{}\ |^q\right)^\frac{1}{q}\ \Bigg|\ L^p(\R^n)\right\Vert \]
%is finite.
\end{defi}

The goal is to trace back everything to the (quasi-)norms of complex-valued functions:
\begin{lem}\hfill\label{trback} \\
i) Let $f=(f_1,\ldots, f_r)^T\in L^p(\R^n, \C^r,w)$. Then, 
\[\Vert f\ | L^p(\R^n, \C^r,w)\Vert \sim \sum_{i=1}^r \Vert f_i\ |L^p(\R^n,w)\Vert.\]
ii) $\Vert \phi\ |\ B_{pq}^s(\R^n, \C^r,w)\Vert\sim\sum_{i=1}^r \Vert \phi_i\ |B^s_{pq}(\R^n,w)\Vert$
\end{lem}

\begin{proof}
i) We abbreviate both $L^p$-quasi-norms with $\Vert .\Vert$. Whether we mean the complex vector valued or just the complex valued function will be clear from the context. From $\Vert f\Vert = \Vert (\sum |f_i|^2)^\frac{1}{2}\Vert$ we obtain with
\[ \sum_i \Vert f_i\Vert\leq r\Vert \sum_i |f_i|\,\Vert\leq r^{\frac{3}{2}}\Vert (\sum_i |f_i|^2)^\frac{1}{2}\Vert\leq r^{\frac{3}{2}}\Vert \sum_i |f_i|\,\Vert\leq r^{\frac{3}{2}}K\sum_i\Vert f_i\Vert\]
the equivalence. Here, the first inequality is obtained from $\Vert f_i\Vert\leq \Vert \sum_i |f_i|\,\Vert$ for each $i$, the second is the Cauchy-Schwarz-inequality, the third is obtained by $\sum |f_i|^2\leq (\sum |f_i|)^2$ and the fourth is the triangle inequality for a quasi-norm with constant $K>1$.\\
ii) As in i) and due to the unconditional convergence, we get 
\begin{align*}
\Vert \phi\ |\ B_{pq}^s&(\R^n, \C^r,w)\Vert= \left(\sum_{j=0}^\infty 2^{jsq}\sum_{i=0}^r \Vert (\alpha_j\hat{\phi_i})\check{}\ |\ L^p(\R^n, \C^r,w)\Vert^q\right)^\frac{1}{q}\\
&\sim \sum_{i=1}^r \left(\sum_{j=0}^\infty 2^{jsq} \Vert (\alpha_j\hat{\phi_i})\check{}\ |\ L^p(\R^n,w)\Vert^q\right)^\frac{1}{q}=\sum_{i=0}^r \Vert \phi_i\ |B^s_{pq}(\R^n,w)\Vert.
\end{align*}
\end{proof}

Thus, the following properties carry over to vector-valued functions:

\begin{rem}\hfill\\
\textbf{i)}\ The Besov spaces are independent of the dyadic resolution of unity $\alpha_j$ \cite[Sect. 2.3.2]{TFS2}.\\
\textbf{ii)}\ If the $B^s_{pq}$-quasi-norm from above is finite, the sum converges unconditionally \cite[Thm. 1.19]{TFS3}.
\end{rem}

\subsection{Spaces on manifolds of bounded geometry}\label{manbgeo}

Let $(M^n,g)$ be a Riemannian manifold of bounded geometry with hermitian vector bundle $E$ of bounded curvature.  We choose a synchronous trivialization of $E$, i.e. on a ball of radius $\epsilon >0$ around $p$ we choose normal coordinates $(x^1,\ldots, x^n)$, trivialize $E$ via parallel transport along radial geodesics and identify $E$ with $\C^r$. 

\begin{lem}\ \!\label{syntr}\cite[Lem. 3.1.6]{Ammha} In the synchronous trivialization described above the following estimates hold for any $\epsilon$ smaller than a certain $\epsilon_0=\epsilon_0(K:=\max(|R|, |R^E|),n, \mathrm{inj})$ and for all $q\in B_\epsilon(p)$:
\[ |\Gamma_{ij}^k(q)|\leq C(n)K \epsilon\]
\[|(\nabla_k\phi)(q)-(\partial_k\phi)(q)|\leq C(n,r)K\epsilon |\phi (q)|\] 
\end{lem}

Denote by $\{ U_\alpha, \xi_\alpha\}$ an open cover of $M$ with synchronous trivializations $\xi_\alpha$ and let $\{\varkappa_\alpha\}$ be a smooth subordinated partition of unity. For simplicity, we assume that the injectivity radius is greater than $1$ and $U_\alpha=B_1(x_\alpha)$.\\
Then, we define \[\Vert \phi\Vert_{H^p_s(E)}^*:= \sum_\alpha \Vert {\xi_\alpha}_* ( \varkappa_\alpha\phi)\Vert_{H^p_s(\R^n,\ \C^r)}\]
where $\Vert ..\Vert_{H^p_s(\R^n,\ \C^r)}$ is defined in \eqref{normE}.\\
Analogously, we obtain the norms for $B^s_{pq}(E)$ and $F^s_{pq}(E)$ in terms of the corresponding spaces on the trivial $\C^r$-bundle on $\R^n$.

\begin{rem}\label{triv-norm}
From Lemma \ref{syntr}, we see that the norms $\Vert .\Vert_{H^p_s(E)}^*$ and $\Vert .\Vert_{H^p_s(E)}$ are equivalent. 
\end{rem}

To obtain Theorem \ref{mainSkrz} on $E$, we can now just follow the lines of the original proof:\\
Let $E$ be trivialized as above. Over a coordinate chart we can assume the bundle be trivial, see Remark \ref{triv-norm}.
From an orthonormal basis $\psi_i$ in $L^2(\R^n)$ we get an orthonormal basis in $L^2(\R^n, \C^r)$ as $\ov{\psi}_{i,k}=(0,\ldots,0,\psi_i,0,\ldots,..,0)$ for $k=1,\ldots, r$. Thus, we obtain a wavelet frame $\ov{\Psi}_{ik}$ by requiring that for all $f\in L^2(E)$ it is $\langle f,\ov{\Psi}_{ik\alpha}\rangle = \langle {\xi_\alpha}_* ( \varkappa_\alpha\phi\circ \exp_{x_\alpha}),\ov{\psi}_{ik}\rangle $, where $\langle ,\rangle$ denote the corresponding $L^2$-products. Then,
\[ \Vert f\ | L^2(E)\Vert \sim \sum_{k=1}^r \left( \sum_{i,\alpha} |\langle f,\ov{\psi}_{ik\alpha}\rangle |^2\right)^{\frac{1}{2}}.\]
Analogously, as in \cite[Thm.1]{Skrz} we obtain an equivalent norm for the weighted Besov spaces in terms of the weighted $L^2$-products $\langle f,\ov{\Psi}_{ik\alpha}\rangle w(x_\alpha)$ where $x_\alpha$ is the center of $U_\alpha$. The only difference to the case of scalar-valued functions is always the finite summation over $k$ which does not affect continuity and compactness. Thus, the proof of \cite{Skrz} simply carries over.\\
Thus, Theorem \ref{mainSkrz} also holds the corresponding spaces on vector bundles of bounded geometry and Theorem \ref{Skrzspin} is then just a special case.

\subsection{Spin manifolds of bounded geometry}

In this section, we will just note some specialties about Sobolev spaces on spinors needed in this article.\\
For that, let $(M^n,g,\sigma)$ be a Riemannian spin manifold. Choose $E=S_g$ the corresponding spinor bundle and let $D=D_g$ be the associated Dirac operator.

\begin{rem}\label{spinE}
If $M$ has bounded geometry, $E=S_g$ also does and from \cite[Sect. 3.1.3]{Ammha} we see that
\[\Vert \phi\ |{H_s^p(M,S)}\Vert^p\sim \sum_{j=0}^s \int_M |\underbrace{D\cdots D}_{j\ times} u|^p\vo_g.\]
\end{rem}

\begin{rem}\label{est}Let $(M,g, \sigma)$ be a spin manifold with bounded geometry.\\
 i) (Inner $L^p$-estimate, \cite[proof of Thm. 8.8 ]{GT}, spin version \cite[proof of Thm. 3.2.1 and 3.2.3]{Ammha}) 
  Let $\phi\in H_{1, loc}^q$ be a solution of $D\phi=\psi$ for $\psi\in L_{loc}^q$. Then,
 for $\epsilon >0$ (smaller than the injectivity radius) there exists a constant $C_\epsilon(q)$ such that for all $x\in M$
 \[ \Vert \phi\Vert_{H_1^q(B_\epsilon(x))}\leq C_\epsilon(q) (\Vert \phi\Vert_{L^q(B_\epsilon(x))} +\Vert \psi\Vert_{L^q(B_\epsilon(x))})\] 
 ii) (Imbedding) Let $n< q$ and $0\leq \gamma\leq 1-\frac{n}{q}$. From a spin version of the proof of \cite[Sect. 7.8 (Thm 7.26)]{GT} we have that for $\epsilon <\mathrm{inj}(M)$ there exists a constant $C$ such that for all $x\in M$
 $H_1^q(B_\epsilon(x))$ is continuously embedded in $ C^{0,\gamma}(\ov{B_\epsilon(x)})$.
 iii) (Schauder estimates) Analogoulsy we have a constant $C(\epsilon,\delta)>0$ such that for $\alpha>0$, $\psi\in C^{0,\alpha}_{loc}$ with $D\phi=\psi$ it holds for all $x\in M$
\[ \Vert \phi\Vert_{C^{1,\alpha}(B_{\epsilon}(x))}\leq C( \Vert \phi\Vert_{C^0(B_{\epsilon+\delta}(x))}+\Vert \psi\Vert_{C^{0,\alpha}(B_{\epsilon+\delta}(x))}.\]
\end{rem}

\begin{cor}\label{cont_emb}
The inner $L^p$-estimates and the imbedding of Remark \ref{est} hold globally on manifold of bounded geometry, i.e.
i) There is a constant $C>0$ such that for $\phi\in H_2^{q}$ and $\psi\in L^q$ with $D\phi=\psi$  it holds
 \[ \Vert \phi\Vert_{H_2^q}\leq C (\Vert \phi\Vert_{L^q} +\Vert \psi\Vert_{L^q}).\]
ii) Let $n< q$ and $0\leq \gamma\leq 1-\frac{n}{q}$. There exists a constant $C$ such that 
$H_1^q$ is continuously embedded in $ C^{0,\gamma}$.
\end{cor}

\begin{proof}
The proof is done analogously as in \cite[Cor. 21]{NG3} by choosing a countable covering of $M$ by geodesic balls ${B_i}$ all of radius $\epsilon < \mathrm{inj} (M)$. Moreover, the covering is chosen such that it is of (finite) multiplicity $m$, i.e. the maximal number of subsets with nonempty intersection is $m$, cf. \cite[Sect. 2.1]{Skrz}.
\end{proof}

\begin{lem}\label{max}
Let $\phi\in H_1^q$ be a solution of $D\phi=c\rho^{\alpha p}|\phi|^{p-2}\phi$ with $\Vert \rho^\alpha \phi\Vert_p=1$ for $p< p_{crit}$. Then, $\limsup_{|x|\to \infty} |\phi(x)|=0$, in particular $|\phi|$ has a maximum.
\end{lem}

The proof is done simultaneously to the corresponding result for the conformal Laplacian \cite[Lem. 22]{NG3}.
 
\end{appendix}

\bibliographystyle{acm}
\bibliography{DirELopen}
Faculty of Mathematics and Computer Science, University Leipzig, \\ Johannisgasse 26,
04103 Leipzig \\ Email-address: grosse@mathematik.uni-leipzig.de
\end{document}